\documentclass{elsarticle}
\usepackage{amsmath}
\usepackage{amssymb}
\usepackage{amsthm}
\usepackage{enumerate}

\newcommand{\nd}{\operatorname{nd}}
\newcommand{\ud}{\operatorname{ud}}
\newcommand{\rank}{\operatorname{rank}}
\newcommand{\norm}[1]{\left\Vert#1\right\Vert}
\newcommand{\diag}{\operatorname{diag}}

\newtheorem{thm}{Theorem}
\newtheorem{prop}{Proposition}
\newtheorem{cor}{Corollary}
\newtheorem{ex}{Example}
\newtheorem{conj}{Conjecture}

\begin{document}

\title{The Normal Defect of Some Classes of Matrices}
\author[Drexel]{Ryan D. Wasson}
\ead{rdw42@drexel.edu}

\author[Drexel]{Hugo J. Woerdeman}
\ead{hugo@math.drexel.edu}

\address[Drexel]{Department of Mathematics, Drexel University, Philadelphia, PA  19104, USA}

\begin{abstract}
An $n \times n$ matrix $A$ has a normal defect of $k$ if there exists an $(n+k) \times (n+k)$ normal matrix $A_{ext}$ with $A$ as a leading principal submatrix and $k$ minimal.  In this paper we compute the normal defect of a special class of $4\times4$ matrices, namely matrices whose only nonzero entries lie on the superdiagonal, and we provide details for constructing minimal normal completion matrices $A_{ext}$.  We also prove a result for a related class of $n\times n$ matrices.  Finally, we present an example of a $6 \times 6$ block diagonal matrix having the property that its normal defect is strictly less than the sum of the normal defects of each of its blocks, and we provide sufficient conditions for when the normal defect of a block diagonal matrix is equal to the sum of the normal defects of each of its blocks.
\end{abstract}


\maketitle

\section{Introduction}

An $n \times n$ matrix $A$ is called \textit{normal} if and only if $A$ commutes with its Hermitian adjoint, i.e., $AA^*=A^*A$.  If $A$ is non-normal, one can construct an $(n+k) \times (n+k)$ normal matrix $A_{ext}$, with $A$ as a leading principal submatrix, called a \textit{normal completion} of $A$.  For example, the matrix $\begin{pmatrix}
A & A^* \\
A^* & A 
\end{pmatrix}$ is a normal completion for any matrix $A$ \cite{PRH}.  An interesting problem first considered in \cite{W} asks the question of how to construct a normal completion of $A$ of smallest possible size (or smallest possible $k$), called a \textit{minimal normal completion} of $A$.  The \textit{normal defect} of $A$, denoted $\nd(A)$, is defined to be the integer $k$ for which $A_{ext} \in \mathbb{C}^{(n+ k) \times (n+k)}$ is of minimal size.  A normal matrix $A$ satisfies $\nd(A)=0$.  Recently, a characterization of matrices with normal defect one was given in \cite{VSW}.

In general, determining the normal defect of a matrix is a challenging problem.  In \cite{KW} (lower bound) and \cite{W} (upper bound), it was shown that the normal defect of a matrix $A$ is bounded by 
\begin{equation} \label{bounds}
\max{ \{i_+[A,A^\ast], i_-[A,A^\ast]\}} \leq \nd(A) \leq \rank(\norm A^2 I_n - A^*A).
\end{equation}  Here $i_\pm(M)$ refers to the number of positive/negative eigenvalues of $M=M^*$.  If these two bounds are equal, then the normal defect of $A$ is immediately known.  However, most of the time this is not the case.

A method for solving this type of problem is to assume that a normal completion of size $(n+\varepsilon(A))\times (n+\varepsilon(A))$ exists, where 
\begin{equation}\label{epsilon definition}
\varepsilon(A):=\max{ \{i_+[A,A^\ast], i_-[A,A^\ast]\}},
\end{equation} and to either solve the system of polynomial equations that results or find a contradiction in the equations.  If there is a contradiction, the process must be repeated with a larger matrix.  This method works but is extremely impractical for large matrices.  For example, using this method on a $4 \times 4$ matrix $A$ satisfying the bounds $2 \leq \nd(A) \leq 3$ requires working with a system of at most $20$ complex variables (or 19 after a simple reduction) when searching for a normal completion of size $6\times 6$.

In this paper, we answer a couple of questions posed in \cite{BW} on the subject of minimal normal completions and explore the consequences of our results.  Specifically, we focus on matrices of a particular form.  In Section \ref{nonzero superdiagonal} we present results on the normal defect of matrices with nonzero entries on the superdiagonal, and in Section \ref{block diagonal} we present results on the normal defect of block diagonal matrices.  We conclude in Section \ref{open questions} with a discussion of open questions that arose during our research.

\section{Matrices with nonzero entries on the superdiagonal} \label{nonzero superdiagonal}

In \cite{KW} (see also \cite[Section 5.9]{BW}), the question was raised as to what the normal defect of a $4\times4$ matrix is whose only nonzero entries lie on the superdiagonal.  A partial answer to this question was given in \cite[Proposition 1]{KW} for the case when the entries are arranged in descending or ascending order by magnitude.  However, it was pointed out in \cite{BW} that the normal defect of 
\begin{equation}\label{0132 matrix}
\begin{pmatrix}  
0 & 1 & 0 & 0 \\
0 & 0 & 3 & 0 \\
0 & 0 & 0 & 2 \\
0 & 0 & 0 & 0 
    \end{pmatrix}
\end{equation} was unknown.  The following result provides the answer to this question.  In addition, the proof of the theorem below provides details for constructing a minimal normal completion matrix.

\begin{thm} \label{thm 0abc matrix}
Let $A \in \mathbb{C}^{4 \times 4}$ be of the form $ A = \begin{pmatrix} 
0 & a & 0 & 0 \\
0 & 0 & b & 0 \\
0 & 0 & 0 & c \\
0 & 0 & 0 & 0 
    \end{pmatrix}$.  Then $\nd(A) = \varepsilon(A)$.  If $A$ is real, then a minimal normal completion can be chosen to be real as well.
    \end{thm}
    
\begin{proof}
We will assume without loss of generality that $a,b,c\in\mathbb{R}$.  If this is not the case, then we are free to make the following transformation.  First, write the complex numbers $a$, $b$, and $c$ in polar form: $a=|a|e^{i\theta_1}$, $b=|b|e^{i\theta_2}$, and $c=|c|e^{i\theta_3}$.  Then, using the fact that $\nd(A)=\nd(UAU^*)$ for any unitary matrix $U$, define $\tilde{A}=UAU^*$ by taking 
\begin{equation*}
U = \begin{pmatrix}
e^{-i(\theta_1+\theta_2+\theta_3)} & 0 & 0 & 0 \\
0 & e^{-i(\theta_2+\theta_3)} & 0 & 0 \\
0 & 0 & e^{-i\theta_3} & 0 \\
0 & 0 & 0 & 1
\end{pmatrix}.
\end{equation*} 
This gives 
\begin{equation*}
\tilde{A} = \begin{pmatrix}
0 & |a| & 0 & 0 \\
0 & 0 & |b| & 0 \\
0 & 0 & 0 & |c| \\
0 & 0 & 0 & 0
\end{pmatrix} 
\end{equation*} whose entries are real numbers.  Hence we can assume that $A$ is real.

The possible values of $\varepsilon(A)$ are 0, 1, 2, and 3.  In each case, the diagonal structure of the commutator matrix $[A,A^\ast]$ puts constraints on the absolute values of $a$, $b$, and $c$.  We will show that $\nd(A)=\varepsilon(A)$ by considering each case separately.  

First, suppose that $\varepsilon(A)=0$.  This can only happen when the commutator  
\begin{equation}  \label{commutator}
[A,A^\ast] = \begin{pmatrix}
|a|^2 & 0 & 0 & 0 \\
0 & |b|^2-|a|^2 & 0 & 0 \\
0 & 0 & -(|b|^2-|c|^2) & 0 \\
0 & 0 & 0 & -|c|^2 
\end{pmatrix}
\end{equation} is the zero matrix.  Thus we must have that $|a| = |b| = |c| = 0$, i.e., $A$ is the zero matrix.  The zero matrix is normal and so $\nd(A) = \varepsilon(A) = 0$.

Next, suppose that $\varepsilon(A)=1$.  By again examining the commutator in (\ref{commutator}), we see that this can only happen under any one of the following conditions:
\begin{enumerate}[(i)]
\item $|a|=|b|=|c| \neq 0$
\item $|a|=0$ and $|b|=|c|\neq0$.
\item $|c| = 0$ and $|b|=|a|\neq0$.
\item Only one of $|a|$, $|b|$, $|c|$ is nonzero.
\end{enumerate}
By 
(\ref{bounds}), we must have that $\nd(A) \geq \varepsilon(A) = 1$.  But in fact, $\nd(A)=\varepsilon(A)=1$ since it is possible to construct normal completion matrices of size $5 \times 5$ in each case.  If $|a| = |b|=|c|$, for example, then 
\begin{equation*}
\left( \begin{array}{cccc|c}
0 & a & 0 & 0 & 0 \\
0 & 0 & b & 0 & 0 \\
0 & 0 & 0 & c & 0 \\
0 & 0 & 0 & 0 & d \\ \hline
f & 0 & 0 & 0 & 0 
\end{array} \right)
\end{equation*} is normal with $|d|=|f|=|b|$.  If instead we had $|a|=0$ and $|b|=|c|\neq0$, then
\begin{equation*}
\left( \begin{array}{cccc|c}
0 & 0 & 0 & 0 & 0 \\
0 & 0 & b & 0 & 0 \\
0 & 0 & 0 & c & 0 \\
0 & 0 & 0 & 0 & d \\ \hline
0 & f & 0 & 0 & 0 
\end{array} \right)
\end{equation*}
is normal with $d$ and $f$ as before.  If $|c|=0$ and $|b|=|a|\neq0$, then 
\begin{equation*}
\left( \begin{array}{cccc|c}
0 & a & 0 & 0 & 0 \\
0 & 0 & b & 0 & 0 \\
0 & 0 & 0 & 0 & d \\
0 & 0 & 0 & 0 & 0 \\ \hline
f & 0 & 0 & 0 & 0 
\end{array} \right)
\end{equation*} is normal with $d$ and $f$ as before.  Finally, in each of the three cases where only one of $|a|$, $|b|$, or $|c|$ is nonzero, 
\begin{equation*}
\left( \begin{array}{cccc|c}
0 & a & 0 & 0 & 0 \\
0 & 0 & 0 & 0 & d \\
0 & 0 & 0 & 0 & 0 \\
0 & 0 & 0 & 0 & 0 \\ \hline
f & 0 & 0 & 0 & 0 
\end{array} \right),
\,\,\,\,
\left( \begin{array}{cccc|c}
0 & 0 & 0 & 0 & 0 \\
0 & 0 & b & 0 & 0 \\
0 & 0 & 0 & 0 & d \\
0 & 0 & 0 & 0 & 0 \\ \hline
0 & f & 0 & 0 & 0 
\end{array} \right),
\,\,\,\,
\left( \begin{array}{cccc|c}
0 & 0 & 0 & 0 & 0 \\
0 & 0 & 0 & 0 & 0 \\
0 & 0 & 0 & c & 0 \\
0 & 0 & 0 & 0 & d \\ \hline
0 & 0 & f & 0 & 0 
\end{array} \right),
\end{equation*}
are each normal, with $|d|=|f|=|a|$, $|d|=|f|=|b|$, and $|d|=|f|=|c|$, respectively.  Note that in each case, the minimal normal completion can be chosen to be real.

Now suppose that $\varepsilon(A)=2$.  By (\ref{commutator}), this can only be true under any one of the following conditions:
\begin{enumerate}[(i)]
\setcounter{enumi}{4}
\item $|b|\geq|a|$ and $|b|\geq|c|$ with both $|a|$ and $|c|$ nonzero and at least one inequality strict.
\item $|b| \geq |a|$ and $|b| \geq |c|$ with exactly one of $|a|$ or $|c|$ equal to zero and both inequalities strict.
\item $|b| \leq |a|$ and $|b| \leq |c|$ with both $|a|$ and $|c|$ nonzero and at least one inequality strict.
\end{enumerate}
The stipulation on the number of strict inequalities is necessary in order to avoid including a few of the conditions found previously for $\varepsilon(A)=1$.  

To show that $\nd(A) = \varepsilon(A)=2$, we construct normal completion matrices of size $6 \times 6$.  Define $A_{ext}\in \mathbb{R}^{6\times6}$ by
    \begin{equation*}
    A_{ext} = \begin{pmatrix}
    A & V \\
    W & Z
    \end{pmatrix}
    \end{equation*}
  with extension matrices $V=(v_{ij}) \in \mathbb{R}^{4\times2}$, $W=(w_{ij})\in\mathbb{R}^{2\times 4}$, and $Z=(z_{ij}) \in \mathbb{R}^{2\times2}$.  If $a$, $b$, and $c$ satisfy the conditions given in (v) or (vi), then $A_{ext}$ can be made to be normal by taking the elements of $V$, $W$, and $Z$ as follows:
  \begin{equation*}
  v_{12} = -a \sqrt{ \frac{(b^2-a^2)(b^2-c^2)}{(bc)^2 + (ba)^2 - (ac)^2}},\,\,\,\,
  v_{31} = -\sqrt{b^2-c^2},
  \end{equation*}
  \begin{equation*}
  v_{42} = \frac{-cb^2 }{\sqrt{(bc)^2 + (ba)^2 - (ac)^2}},\,\,\,\,
  w_{12} = \sqrt{b^2-a^2},
  \end{equation*}
  \begin{equation*}
  w_{21} = \frac{ab^2}{\sqrt{(bc)^2 + (ba)^2 - (ac)^2}},\,\,\,\,
  w_{24} = c \sqrt{ \frac{(b^2-a^2)(b^2-c^2)}{(bc)^2 + (ba)^2 - (ac)^2}},
  \end{equation*}
  \begin{equation*}
  z_{12} = a^2 \sqrt{\frac{b^2-c^2}{(bc)^2 + (ba)^2 - (ac)^2}},\,\,\,\,
  z_{21} = c^2 \sqrt{\frac{b^2-a^2}{(bc)^2 + (ba)^2 - (ac)^2}},
  \end{equation*}
  with all other elements not listed above set to zero.  Thus in this case $A_{ext}$ has the form
  \begin{equation}\label{b_largest}
  A_{ext} = \left( \begin{array}{cccc|cc}
0 & a & 0 & 0 & 0 & v_{12} \\
0 & 0 & b & 0 & 0 & 0 \\
0 & 0 & 0 & c & v_{31} & 0 \\
0 & 0 & 0 & 0 & 0 & v_{42} \\ \hline
0 & w_{12} & 0 & 0 & 0 & z_{12} \\
w_{21} & 0 & 0 & w_{24} & z_{21} & 0 
\end{array} \right).
\end{equation}
The conditions in (v) and (vi) guarantee that each element in $A_{ext}$ is real.  Although it is a somewhat tedious algebraic exercise, it can easily be verified that $A_{ext}$ is normal (for details, see the Appendix).

Now suppose instead that $a$, $b$, and $c$ satisfy the conditions given in (vii).  In this case there are actually two additional possibilities.  It must be true that either 
\begin{equation}\label{vii.a,vii.b}
\textnormal{(vii.a)\,\,\,\,}
c^2 > \frac{a^4}{2a^2-b^2}
\textnormal{\,\,\,\,\,\,\,or\,\,\,\,\,\,\,}
\textnormal{(vii.b)\,\,\,\,}
a^2 > \frac{c^4}{2c^2-b^2}
\end{equation}
since the falseness of one implies that the other is true.  To see this, assume (vii.a) is false.  Then, using $|b| \leq |a|$,
\begin{equation}\label{c leq a}
c^2 \leq \frac{a^4}{2a^2-b^2} \leq \frac{a^4}{2a^2-a^2} = a^2
\end{equation}
and since $|b| \leq |c|$, we have that
\begin{equation*}
c^2(c^2-b^2) \geq 0 \,\,\,\, \Rightarrow \,\,\,\, c^4-(cb)^2 = 2c^4 - (cb)^2 - c^4 = c^2(2c^2-b^2)-c^4\geq 0
\end{equation*}
\begin{equation} \label{a leq c}
\Rightarrow \,\,\,\, c^2 \geq \frac{c^4}{2c^2-b^2}.
\end{equation}
But $a^2 \geq c^2$ by (\ref{c leq a}), and so 
\begin{equation*}
a^2 > \frac{c^4}{2c^2-b^2}.
\end{equation*}
This last inequality is strict because one of the inequalities in (\ref{c leq a}) or (\ref{a leq c}) must be strict since we are not considering the case $|a|=|b|=|c|$. The other direction is equivalent.  Note that the denominators in both conditions cannot equal 0 since $|b| \leq |a|$, $|b| \leq |c|$, and $a$, $c \neq 0$.

If $a$, $b$, and $c$ satisfy the conditions given in (vii) and (vii.a), then $A_{ext}$ can be made to be normal by taking
\begin{equation*}
v_{12} = \frac{a(c^2-a^2)}{\sqrt{\beta}},\,\,\,\,
v_{22}=c\sqrt{\frac{(a^2-b^2)(c^2-b^2)}{\beta}},\,\,\,\,
v_{41}=c,
\end{equation*}
\begin{equation*}
w_{11}=-\frac{bc}{a},\,\,\,\,
w_{12}=\frac{b(a^2-c^2)}{a^2} \sqrt{\frac{a^2-b^2}{c^2-b^2}},\,\,\,\,
w_{13}=\frac{c(b^2-a^2)}{a^2},
\end{equation*}
\begin{equation*}
w_{21}=\frac{c^2(b^2-a^2)}{a\sqrt{\beta}},\,\,\,\,
w_{22}=\frac{c(a^2-b^2)(a^2-c^2)}{a^2} \sqrt{\frac{a^2-b^2}{\beta(c^2-b^2)}},\,\,\,\,
w_{23}=\frac{b\sqrt{\beta}}{a^2},
\end{equation*}
\begin{equation*}
z_{12}=\frac{b}{a^2} \sqrt{\frac{\beta(a^2-b^2)}{c^2-b^2}},\,\,\,\,
z_{22}=-\frac{c}{\beta a^2} (a^6+3(abc)^2-c^2b^4-c^2a^4-2b^2a^4) \sqrt{\frac{a^2-b^2}{c^2-b^2}}
\end{equation*}
where $\beta= 2(ac)^2-(bc)^2-a^4$ and with all elements not listed set to zero.  Then $A_{ext}$ has the form
\begin{equation}\label{c_largest}
  A_{ext} = \left( \begin{array}{cccc|cc}
0 & a & 0 & 0 & 0 & v_{12} \\
0 & 0 & b & 0 & 0 & v_{22} \\
0 & 0 & 0 & c & 0 & 0 \\
0 & 0 & 0 & 0 & v_{41} & 0 \\ \hline
w_{11} & w_{12} & w_{13} & 0 & 0 & z_{12} \\
w_{21} & w_{22} & w_{23} & 0 & 0 & z_{22} 
\end{array} \right).
\end{equation}

If instead $a$, $b$, and $c$ satisfy the conditions given in (vii) and (vii.b), then $A_{ext}$ can be made to be normal by taking
\begin{equation*}
v_{21}=\frac{b\sqrt{\beta}}{c^2},\,\,\,\,
v_{22}=\frac{a(b^2-c^2)}{c^2},\,\,\,\,
v_{31}=\frac{a(c^2-b^2)(c^2-a^2)}{c^2} \sqrt{\frac{c^2-b^2}{\beta(a^2-b^2)}},
\end{equation*}
\begin{equation*}
v_{32}=\frac{b(c^2-a^2)}{c^2} \sqrt{\frac{c^2-b^2}{a^2-b^2}},\,\,\,\,
v_{41}=\frac{a^2(b^2-c^2)}{c\sqrt{\beta}},\,\,\,\,
v_{42}=-\frac{ab}{c},
\end{equation*}
\begin{equation*}
w_{13}=a\sqrt{\frac{(a^2-b^2)(c^2-b^2)}{\beta}},\,\,\,\,
w_{14} = \frac{c(a^2-c^2)}{\sqrt{\beta}},\,\,\,\,
w_{21}=a,
\end{equation*}
\begin{equation*}
z_{11}=-\frac{a}{\beta c^2} (c^6+3(abc)^2-a^2b^4-a^2c^4-2b^2c^4) \sqrt{\frac{c^2-b^2}{a^2-b^2}},\,\,\,\,
z_{12}=\frac{b}{c^2} \sqrt{\frac{\beta(c^2-b^2)}{a^2-b^2}}
\end{equation*}
where $\beta= 2(ac)^2-(ab)^2-c^4$ and with all elements not listed set to zero.  Then $A_{ext}$ has the form
\begin{equation}\label{a_largest}
  A_{ext} = \left( \begin{array}{cccc|cc}
0 & a & 0 & 0 & 0 & 0 \\
0 & 0 & b & 0 & v_{21} & v_{22} \\
0 & 0 & 0 & c & v_{31} & v_{32} \\
0 & 0 & 0 & 0 & v_{41} & v_{42} \\ \hline
0 & 0 & w_{13} & w_{14} & z_{11} & z_{12} \\
w_{21} & 0 & 0 & 0 & 0 & 0 
\end{array} \right).
\end{equation}

In either case (\ref{c_largest}) or (\ref{a_largest}), the denominators of the elements of $V$, $W$, and $Z$ are never zero.  The number $\beta$ is never less than or equal to zero because of the conditions in (\ref{vii.a,vii.b}), and $a$ and $c$ were assumed to be nonzero from the start.  The conditions in (vii) guarantee that each element is real.  As was the case for (\ref{b_largest}), it can easily be verified that (\ref{c_largest}) and (\ref{a_largest}) are normal, although the algebra is very tedious.

It is interesting to note that the formulas given for $V$, $W$, and $Z$ in (\ref{c_largest}) and (\ref{a_largest}) are related via the change of variables $a \rightarrow c$, $c \rightarrow a$, and the equations
\begin{equation*}
V_{b}=
\begin{pmatrix}
0 & 0 & 0 & 1 \\
0 & 0 & 1 & 0 \\
0 & 1 & 0 & 0 \\
1 & 0 & 0 & 0
\end{pmatrix} W_{a}^T
\begin{pmatrix}
0 & 1 \\
1 & 0 
\end{pmatrix},
\end{equation*}
\begin{equation*}
W_b=
\begin{pmatrix}
0 & 1 \\
1 & 0 
\end{pmatrix} V_{a}^T
 \begin{pmatrix}
0 & 0 & 0 & 1 \\
0 & 0 & 1 & 0 \\
0 & 1 & 0 & 0 \\
1 & 0 & 0 & 0
\end{pmatrix},
\end{equation*}
\begin{equation*}
Z_b=
\begin{pmatrix}
0 & 1 \\
1 & 0 
\end{pmatrix} Z_a^T
\begin{pmatrix}
0 & 1 \\
1 & 0 
\end{pmatrix},
\end{equation*}
where $V_a$, $W_a$, $Z_a$ and $V_b$, $W_b$, $Z_b$ are the extension matrices in (\ref{c_largest}) and (\ref{a_largest}), respectively.

Finally, suppose that $\varepsilon(A)=3$.  By (\ref{commutator}), the only two possible conditions on $a$, $b$, and $c$ are $0 < |a| < |b| < |c|$ or $|a| > |b| > |c| > 0$.  In either case, the normal defect was previously shown to be 3 in \cite[Proposition 1]{KW} since $\varepsilon(A)=\ud(A)=3$.  Examples of real minimal normal completions of $A$ are given by 
\begin{equation*}
 \left( \begin{array}{cccc|ccc}
0 & a & 0 & 0 & 0 & 0 & -\sqrt{c^2-a^2} \\
0 & 0 & b & 0 & -\sqrt{c^2-b^2} & 0 & 0 \\
0 & 0 & 0 & c & 0 & 0 & 0 \\
0 & 0 & 0 & 0 & 0 & c & 0 \\ \hline
0 & \sqrt{c^2-a^2} & 0 & 0 & 0 & 0 & a \\
c & 0 & 0 & 0 & 0 & 0 & 0 \\
0 & 0 & \sqrt{c^2-b^2} & 0 & b & 0 & 0 
\end{array} \right)
\end{equation*}
for the case $0 < |a| < |b| < |c|$, and 
\begin{equation*}
 \left( \begin{array}{cccc|ccc}
0 & a & 0 & 0 & 0 & 0 & 0 \\
0 & 0 & b & 0 & -\sqrt{a^2-b^2} & 0 & 0 \\
0 & 0 & 0 & c & 0 & -\sqrt{a^2-c^2} & 0 \\
0 & 0 & 0 & 0 & 0 & 0 & a \\ \hline
0 & 0 & 0 & \sqrt{a^2-c^2} & 0 & c & 0 \\
a & 0 & 0 & 0 & 0 & 0 & 0 \\
0 & 0 & \sqrt{a^2-b^2} & 0 & b & 0 & 0 
\end{array} \right)
\end{equation*}
for the case $|a| > |b| > |c| > 0$.

Thus, we have shown that $\nd(A) = \varepsilon(A)$ and that a minimal normal completion of $A$ can be chosen to be real if $A$ is real.

\end{proof}

\begin{ex} \label{ex 0132 matrix}
\rm{
We can now write down a minimal normal completion matrix for (\ref{0132 matrix}) using the details of the proof of Theorem \ref{thm 0abc matrix}.  This matrix satisfies the conditions in case (v) and so we will construct a normal completion matrix according to (\ref{b_largest}).  This gives
\begin{equation*}
\left( \begin{array}{cccc|cc}
0 & 1 & 0 & 0 & 0 & -\frac{2}{41} \sqrt{410} \\
0 & 0 & 3 & 0 & 0 & 0 \\
0 & 0 & 0 & 2 & -\sqrt{5} & 0 \\
0 & 0 & 0 & 0 & 0 & -\frac{18}{41} \sqrt{41} \\ \hline
0 & 2\sqrt{2} & 0 & 0 & 0 & \frac{1}{41}\sqrt{205} \\
\frac{9}{41}\sqrt{41} & 0 & 0 & \frac{4}{41}\sqrt{410} & \frac{8}{41}\sqrt{82} & 0 
\end{array} \right)
\end{equation*}
and so (\ref{0132 matrix}) has normal defect 2.
} 
\end{ex}

In the following three special cases, it can be shown that a minimal normal completion matrix of the $n \times n$ matrices
\begin{equation*}
\begin{pmatrix}
0 & a & \hdots & 0 \\ 
 \vdots   & \vdots   & bI_{n-3} & \vdots \\ 
 0   &  0   &   \hdots      & c \\
0  & 0 & \hdots & 0
\end{pmatrix}, \,\,\,\, |b| \geq |a|, |c|
\end{equation*}
\begin{equation*}
\begin{pmatrix}
0 & a & 0 & \hdots \\ 
 0   & 0   & b & \hdots \\ 
\vdots   &  \vdots   &   \vdots      & cI_{n-3} \\
0  & 0 & 0 & \hdots
\end{pmatrix}, \,\,\,\, |b| \leq |a|, |c| \,\,\,\textnormal{and } |c| \geq |a|
\end{equation*}
\begin{equation*}
\begin{pmatrix}
\vdots & aI_{n-3} & \vdots & \vdots \\ 
0   & \hdots   & b & 0 \\ 
 0   &  \hdots   &   0      & c \\
0  & \hdots & 0 & 0
\end{pmatrix}, \,\,\,\, |b| \leq |a|, |c| \,\,\, \textnormal{and  } |a| \geq |c|
\end{equation*}
follows the exact same form as (\ref{b_largest}), (\ref{c_largest}), and (\ref{a_largest}).  In general, however, it is not clear whether the results of Theorem \ref{thm 0abc matrix} can be extended to $n\times n$ matrices whose only nonzero entries lie on the diagonal.  For example, the $5\times 5$ matrix 
\begin{equation} \label{nd unknown}
A=\begin{pmatrix}
0 & 2 & 0 & 0 & 0 \\
0 & 0 & 1 & 0 & 0 \\
0 & 0 & 0 & 1 & 0 \\
0 & 0 & 0 & 0 & 1 \\
0 & 0 & 0 & 0 & 0
\end{pmatrix}
\end{equation} has $\varepsilon(A)=2$, but the normal defect of $A$ is unknown.  To show that $\nd(A)=2$, it is necessary to assume that a normal completion of size $7 \times 7$ exists and to solve the equations that result.  It is easy to write down a normal completion of size $8\times 8$ (see Example \ref{ex open question}), but to show that $\nd(A) = 3$, it is necessary to find a contradiction in the equations that result from assuming $\nd(A)=2$.  We suspect that $\nd(A)=3$, but in either case it is challenging.

It is known, however, that the equality $\nd(A)=\varepsilon(A)$ in Theorem \ref{thm 0abc matrix} can \textit{not} be extended in general to $n \times n$ matrices of the form 
\begin{equation}\label{nxn n nonzero entries}
A = \begin{pmatrix} 
0 & a_1 &  & \cdots & 0  \\
  &  & a_2 &  &  \\
  \vdots &  & & \ddots & \vdots \\
& & & & a_{n-1} \\
a_n & & & \cdots & 0
    \end{pmatrix},
\end{equation} when $n\geq4$.  An example of this fact was given in \cite[Section 5.9]{BW} for the transpose of the $4\times 4$ matrix 
\begin{equation*}
\begin{pmatrix}
0 & 1 & 0 & 0 \\
0 & 0 & 1 & 0 \\
0 & 0 & 0 & 1 \\
\sqrt{2} & 0 & 0 & 0
\end{pmatrix}
\end{equation*}
which has the property that $\nd(A) > \varepsilon(A)=1$.  Interestingly, this matrix also has the property that both bounds in (\ref{bounds}) are strict.  Indeed, $\rank(\norm A^2 I_n - A^*A)=3$, but $\nd(A)=2$ since a minimal normal completion is given by
\begin{equation*}
\left( \begin{array}{cccc|cc}
0 & 1 & 0 & 0 & 0 & -1 \\
0 & 0 & 1 & 0 & 1 & 0 \\
0 & 0 & 0 & 1 & 0 & 0 \\
\sqrt{2} & 0 & 0 & 0 & 0 & 0 \\ \hline
0 & 0 & 0 & 1 & 0 & 0 \\
0 & 1 & 0 & 0 & 0 & 1
\end{array} \right).
\end{equation*}
The following proposition 
identifies necessary and sufficient conditions on the entries $a_1$, $a_2$, ..., $a_n$ for when $\nd(A)=\varepsilon(A)=1$.  

\begin{prop}\label{prop nd = max}
Let $A \in \mathbb{C}^{n\times n}$ be of the form $(\ref{nxn n nonzero entries})$, with $n\geq4$.  Define the set $A_x=\{k \,(\textnormal{mod }n) : |a_k| = |x|\}$.  Then $\nd(A) = \varepsilon(A) = 1$ if and only if there exist $\alpha, \beta \in \mathbb{R}$ such that $\alpha > \beta \geq 0$ with $|a_k| \in \{\alpha,\beta \}$ for all $k \leq n$ and there exist $i,j \in \mathbb{N}$ such that $A_\beta = \{i,i+1,... ,i+j-1\}$, where
    \begin{enumerate}[(i)]
    \item $1 \leq j \leq 2$ (modulo $n$) if $\beta \neq 0$
    \item $1 \leq j \leq n-1$ (modulo $n$) if $\beta=0$
    \end{enumerate}
\end{prop}

Note that cases (i), (ii), (iii), and (iv) in the beginning of the proof of Theorem \ref{thm 0abc matrix} satisfy the conditions of this proposition with $\beta = 0$.

\begin{proof}
``$\Leftarrow$":  The only nonzero entries of the commutator 
\begin{equation*}
[A,A^*]=\begin{pmatrix}
|a_1|^2-|a_n|^2 &  & \cdots & 0 \\
 & |a_2|^2-|a_1|^2 & & \vdots \\
\vdots & & \ddots & \\
0 & \cdots & & |a_n|^2-|a_{n-1}|^2
\end{pmatrix}
\end{equation*} occur at the positions $(i,i)$ and $(i+j,i+j)$ with values $\beta^2-\alpha^2$ and $\alpha^2-\beta^2$, respectively.  Since $\alpha > \beta$ we get that $\varepsilon(A)=1$.  

Next, take $x = \sqrt{\alpha^2-\beta^2}e_{i+j}$ and $y = \sqrt{\alpha^2-\beta^2}e_i$ (where $e_i$ and $e_{i+j}$ are the standard basis vectors).  These vectors are linearly independent and satisfy the equality $[A,A^*]=xx^*-yy^*$.  To show that $\nd(A)=1$, we must show that $x$, $y$, $Ax$, and $A^*y$ are linearly dependent \cite[Theorem 5.9.4 (ii)]{BW}.

The matrix $A$ has the property that $A(e_k)=a_{k-1}e_{k-1}$ and $A^*(e_k)=a_ke_{k+1}$ for any positive integer $k \leq n$.  Hence, $Ax = a_{i+j-1} \sqrt{\alpha^2-\beta^2} e_{i+j-1}$ and $A^*y = a_{i} \sqrt{\alpha^2-\beta^2} e_{i+1}$.  If $\beta = 0$, then $|a_{i+j-1}|=|a_i|=0$ (since $i, i+j-1 \in A_\beta$ by assumption) and so $Ax=A^*y$.  If $\beta \neq 0$, then either $A^*y = a_i x$ (if $j = 1$) or $A^*y = (a_i / a_{i+j-1}) (Ax)$ (if $j = 2$).  In either case, the vectors $x$, $y$, $Ax$, and $A^*y$ are linearly dependent, and so $\nd(A)=1$.  A minimal normal completion matrix is given by $A_{ext}=\begin{pmatrix}
A & y \\
x^* & 0
\end{pmatrix}$.
\\
\\
``$\Rightarrow$":  Assume $\nd(A) = \varepsilon(A)=1$.  Since $\nd(A)=1$, there exist linearly independent $x,y\in \mathbb{C}^n$ such that $[A,A^*]=xx^*-yy^*$ and $x$, $y$, $Ax$, and $A^*y$ are linearly dependent.  Since $\varepsilon(A)=1$ and $[A,A^*]$ is diagonal, there exist $i,j\in \mathbb{N}$ such that the entries at the positions $(i,i)$ and $(i+j,i+j)$ of the commutator $[A,A^*]$ are nonzero (if only one entry was nonzero, then $x$ and $y$ would be linearly dependent).  Without loss of generality assume that the $(i,i)$ entry is negative and the $(i+j,i+j)$ entry is positive.  The values at these positions are $|a_i|^2-|a_{i-1}|^2$ and $|a_{i+j}|^2-|a_{i+j-1}|^2$, respectively.

Since these are the only two nonzero entries, we must have $|a_i| = |a_{i+1}| = \cdots = |a_{i+j-1}| = \beta$ for some $\beta \in \mathbb{R}$.  All other elements $a_k$ of the matrix $A$ must have the property that $|a_k| = \alpha$ for some $\alpha \in \mathbb{R}$.  Thus, $A_\beta = \{i,i+1,...,i+j-1\}$ and $|a_k| \in \{\alpha,\beta\}$ for all $k \leq n$.  Since the $(i,i)$ entry is negative and the $(i+j,i+j)$ entry is positive, we must have $\alpha > \beta \geq 0$.  In addition, note that we cannot have $j = 0$ (mod n); otherwise, $A$ would be a normal matrix.  Thus, if $\beta=0$, we know that $j$ satisfies the constraint (ii).

If $\beta \neq 0$, take $x = \sqrt{\alpha^2-\beta^2}e_{i+j}$ and $y = \sqrt{\alpha^2-\beta^2}e_i$ so that $[A,A^*]=xx^*-yy^*$ as before.  Consider the set of vectors $\{x,y,Ax+tAy,A^*y+\overline{t}A^*x\}$, where $t$ is a complex number satisfying $|t|<1$.  Ignoring coefficients, this set is equivalent to the set $V = \{e_{i+j},e_i,e_{i+j-1}+te_{i-1}, e_{i+1}+\overline{t}e_{i+j+1}\}$.  Assume for the sake of contradiction that $j \geq 3$.  First, observe that the vectors $\{e_{i+j},e_i,e_{i+j-1},e_{i+1}\}$ must be linearly independent.  Otherwise, at least two of them are equal.  But this would contradict our assumption that $j \geq 3$ since
\begin{equation*}
e_{i+j}=e_{i}\,\,\,\,\Rightarrow \,\,\,\,i+j \equiv i \mod{n} \,\,\,\,\Rightarrow\,\,\,\,j \equiv 0 \mod{n}
\end{equation*}
\begin{equation*}
e_{i+j}=e_{i+1} \,\,\,\, \Rightarrow \,\,\,\, i+j \equiv i+1 \mod{n} \,\,\,\,\Rightarrow\,\,\,\, j \equiv 1 \mod{n}
\end{equation*}
\begin{equation*}
e_{i+j-1}=e_{i} \,\,\,\, \Rightarrow \,\,\,\, i+j-1 \equiv i \mod{n} \,\,\,\, \Rightarrow \,\,\,\, j \equiv 1 \mod{n}
\end{equation*}
\begin{equation*}
e_{i+j-1}=e_{i+1} \,\,\,\,\Rightarrow \,\,\,\,i+j-1 \equiv i+1 \mod{n} \,\,\,\,\Rightarrow\,\,\,\, j \equiv 2 \mod{n}.
\end{equation*}
Although not listed in the equations above, the other two trivial possibilities ($e_{i+j}=e_{i+j-1}$ and $e_i=e_{i+1}$) are clearly impossible.  Thus, these four vectors are linearly independent.

This implies that the vectors in $V$ are also linearly independent, since we can form the matrix \begin{equation*}
M = \begin{pmatrix}
e_{i+j} & e_i & e_{i+j-1}+te_{i-1} & e_{i+1}+\overline{t}e_{i+j+1}
\end{pmatrix}
\end{equation*}
(whose columns are the vectors in $V$) and use row reduction to eliminate the terms $te_{i-1}$ and $\overline{t}e_{i+j+1}$.  Hence, if $j\geq 3$, we have shown that there exists $x,y\in \mathbb{C}^n$ such that $[A,A^*]=xx^*-yy^*$ and that for all $t \in \mathbb{C}$ with $|t| < 1$ the vectors $x$, $y$, $Ax+tAy$, and $A^*y+\overline{t}A^*x$ are linearly independent.  According to \cite[Theorem 5.9.4 (iii)]{BW}, this contradicts the fact that $\nd(A)=1$.  Thus, $j$ must satisfy the constraint (i).

\end{proof}

The usefulness of this proposition lies in its ability to identify cases when a matrix $A$ satisfies $\nd(A) >  \varepsilon(A) = 1$.  The next example illustrates this fact.

\begin{ex} \label{ex nd > max}
\rm {
The $5\times 5$ matrix
\begin{equation*}
A = \begin{pmatrix}
0 & -2 & 0 & 0 & 0 \\
0 & 0 & 1 & 0 & 0 \\
0 & 0 & 0 & -1 & 0 \\
0 & 0 & 0 & 0 & i \\
2 & 0 & 0 & 0 & 0 \\
\end{pmatrix}
\end{equation*}
has $\nd(A) > \varepsilon(A)= 1$.  To see this, take $\alpha = 2$ and $\beta = 1$.  Then $A_\beta=\{i,i+1,i+j-1\}$ with $i=2$ and $j=3$.  Since we have $j > 2$, this matrix fails to satisfy the hypotheses of Proposition \ref{prop nd = max}.  Thus, either $\nd(A) \neq 1$ or $\varepsilon(A) \neq 1$.  It is easy to show that this matrix satisfies $\varepsilon(A) = 1$; hence, it must be the case that $\nd(A) > 1$.
}
\end{ex}

\section{Normal defect of block diagonal matrices} \label{block diagonal}

Another question raised in \cite[Section 5.9]{BW} asks whether or not it holds in general that the normal defect of a block diagonal matrix is equal to the sum of the normal defects of each block.  In other words, does the equality 
\begin{equation}\label{block equality}
\nd(\diag(A_i)_{i=1}^m)=\sum_{i=1}^m \nd(A_i)
\end{equation} hold for all square matrices $A_i$?  The answer to this question is no, as the following example shows.

\begin{ex}\label{ex block}
\rm{
Consider the $6\times6$ block diagonal matrix $A = \begin{pmatrix} 
A_1 &  \\
        & A_2 \\
    \end{pmatrix}$ with $A_1=  \begin{pmatrix} 
0 & 1 & 0 \\
0 & 0 & 2  \\
0 & 0 & 0  \\
    \end{pmatrix}$ 
and $A_2=A_1^T$.  In this case, $\nd(A_1)+\nd(A_2)=2+2=4$ since $[A_1,A_1^*]=\begin{pmatrix} 
1 & 0 & 0 \\
0 & 3 & 0  \\
0 & 0 & -4  \\
    \end{pmatrix}$ and $[A_2,A_2^*]=\begin{pmatrix} 
-1 & 0 & 0 \\
0 & -3 & 0  \\
0 & 0 & 4  \\
    \end{pmatrix}$ and $\nd(B)=\varepsilon(B)$ for any $3\times3$ matrix $B$ \cite[Corollary 5.9.7]{BW}.  But $\nd(A)=3$ since it is possible to construct a normal completion matrix of size $9\times9$.  For example,
\begin{equation*}
\left( \begin{array}{ccc|ccc|ccc}
0 & 1 & 0 & 0 & 0 & 0 & 0 & 0 & -\frac{3}{\sqrt{7}} \\
0 & 0 & 2 & 0 & 0 & 0 & 0 & 0 & 0 \\
0 & 0 & 0 & 0 & 0 & 0 & 0 & 2 & 0 \\ \hline
0 & 0 & 0 & 0 & 0 & 0 & 0 & 0 & -\frac{4}{\sqrt{7}} \\
0 & 0 & 0 & 1 & 0 & 0 & -\sqrt{3} & 0 & 0 \\
0 & 0 & 0 & 0 & 2 & 0 & 0 & 0 & 0 \\ \hline
0 & \sqrt{3} & 0 & 0 & 0 & 0 & 0 & 0 & \sqrt{\frac{3}{7}} \\
\frac{4}{\sqrt{7}} & 0 & 0 & \frac{3}{\sqrt{7}} & 0 & 0 & \sqrt{\frac{3}{7}} & 0 & 0 \\
0 & 0 & 0 & 0 & 0 & 2 & 0 & 0 & 0 
\end{array} \right)
\end{equation*}
is normal.  Thus $\nd(A)<\nd(A_1)+\nd(A_2)$.  
}
\end{ex}

To understand the reason for this inequality, observe that the commutators of $A_1$ and $A_2$ have a different number of positive and negative eigenvalues.  Specifically, $i_+[A_1,A_1^*] > i_-[A_1,A_1^*]$ but $i_+[A_2,A_2^*]<i_-[A_2,A_2^*]$.  We can use this example to come up with sufficient conditions for when the equality in (\ref{block equality}) holds.  Recall from (\ref{epsilon definition}) that $\varepsilon(A):=\max{ \{i_+[A,A^\ast], i_-[A,A^\ast]\}}$.

\begin{prop}\label{prop block}
Let $A=\diag(A_i)_{i=1}^m$ be a block diagonal matrix with each $A_i \in \mathbb{C}^{n_i \times n_i}$ and each $n_i\in \mathbb{N}$.  Then $\nd(A) = \sum_{i=1}^m \nd(A_i)=\varepsilon(A)$ if and only if $\nd(A_i) = \varepsilon(A_i)$ and either $i_+[A_i,A_i^*] \geq i_-[A_i,A_i^*]$ or $i_+[A_i,A_i^*] \leq i_-[A_i,A_i^*]$ for each $i \leq m$.
\end{prop}

\begin{proof}
``$\Leftarrow$":  For each $i \leq m$, assume that $\nd(A_i) = \varepsilon(A_i)$, and without loss of generality assume that $i_+[A_i,A_i^*] \geq i_-[A_i,A_i^*]$.  Observe that 
\begin{equation*}
[A,A^*] = \begin{pmatrix} 
[A_1,A_1^*] & & \\
& \ddots &  \\
&  & [A_m,A_m^*]  \\
    \end{pmatrix},
\end{equation*}  and so the eigenvalues of $[A,A^*]$ are the eigenvalues of each block $[A_i,A_i^*]$.  Since $i_+[A_i,A_i^*] \geq i_-[A_i,A_i^*]$, we get that $\varepsilon(A)=\sum_{i=1}^m \varepsilon(A_i)$.  In addition, $\nd(A_i)=\varepsilon(A_i)$ implies that $\sum_{i=1}^m \nd(A_i) = \sum_{i=1}^m \varepsilon(A_i)$.  Thus:
\begin{equation*}
\sum_{i=1}^m \varepsilon(A_i) = \varepsilon(A) \leq \nd(A) \leq \sum_{i=1}^m \nd(A_i) = \sum_{i=1}^m \varepsilon(A_i),
\end{equation*} and so we have $\nd(A) = \sum_{i=1}^m \nd(A_i) = \varepsilon(A)$ as required.
\\
\\
``$\Rightarrow$":  Assume that $\nd(A) = \sum_{i=1}^m \nd(A_i) = \varepsilon(A)$.  In general it must be true that $\varepsilon(A) \leq \sum_{i=1}^m \varepsilon(A_i)$.  But it is also the case that $\nd(A_i) \geq \varepsilon(A_i)$ for each $i \leq k$.  Thus:
\begin{equation}\label{nd inequalities}
\sum_{i=1}^m \varepsilon(A_i) \leq \sum_{i=1}^m \nd(A_i) = \nd(A) = \varepsilon(A) \leq \sum_{i=1}^m \varepsilon(A_i).
\end{equation}
This implies that $\varepsilon(A)=\sum_{i=1}^m \varepsilon(A_i)$, which can only be possible when either $i_+[A_i,A_i^*] \geq i_-[A_i,A_i^*]$ or $i_+[A_i,A_i^*] \leq i_-[A_i,A_i^*]$, for each $i \leq m$..  Additionally, (\ref{nd inequalities}) implies that $\sum_{i=1}^m \nd(A_i) = \sum_{i=1}^m \varepsilon(A_i)$, which can only be possible when $\nd(A_i)=\varepsilon(A_i)$ since $\nd(A_i)\geq \varepsilon(A_i)$ in general.

\end{proof}

An immediate consequence of Proposition \ref{prop block} is presented in the following corollary:

\begin{cor}
Let $A=\diag(A_i)_{i=1}^m$ be a block diagonal matrix consisting of at least $m-1$ blocks of size $2\times2$ and at most $1$ block of size $3\times3$.  Then $\nd(A) = \sum_{i=1}^m \nd(A_i)$.
\end{cor}

\begin{proof}
Any matrix with this form satisfies the hypotheses of Proposition \ref{prop block}.
\end{proof}

The conditions given in Proposition \ref{prop block} are necessary and sufficient conditions for the equality $\nd(A) = \sum_{i=1}^m \nd(A_i) = \varepsilon(A)$ to hold (with $A$ being block diagonal).  However, it should be emphasized that these are only sufficient conditions for the equality $\nd(A)=\sum_{i=1}^m \nd(A_i)$ (by itself) to hold.  Assuming that $\nd(A_i) = \varepsilon(A_i)$ and either $i_+[A_i,A_i^*] \geq i_-[A_i,A_i^*]$ or $i_+[A_i,A_i^*] \leq i_-[A_i,A_i^*]$ for each $i \leq m$ is enough guarantee that $\nd(A) = \sum_{i=1}^m \nd(A_i)$, but it has the additional consequence that $\nd(A)=\varepsilon(A)$.  One could ask whether there exist matrices satisfying $\nd(A) = \sum_{i=1}^m \nd(A_i) \neq \varepsilon(A)$.  This can certainly happen (trivially) when one or more of the blocks are normal.  But the question remains whether or not there exist non-trivial examples of block diagonal matrices satisfying $\nd(A) = \sum_{i=1}^m \nd(A_i) \neq \varepsilon(A)$.  A possible candidate is 
\begin{equation*}
\left( \begin{array}{cc|cccc}
0 & 1 & 0 & 0 & 0 & 0  \\
0 & 0 & 0 & 0 & 0 & 0  \\ \hline
0 & 0 & 0 & 1 & 0 & 0  \\
0 & 0 & 0 & 0 & 1 & 0 \\
0 & 0 & 0 & 0 & 0 & 1  \\
0 & 0 & 2 & 0 & 0 & 0 \\ 
\end{array} \right),
\end{equation*} but the normal defect of this matrix is unknown.

\section{Observations and open questions} \label{open questions}

It would be interesting to know if the normal defects of the $n \times n$ and $(n+1) \times (n+1)$ matrices
\begin{equation}\label{two matrices different size}
A = \begin{pmatrix}
0 & a_1 &  & \cdots & 0  \\
  &  & a_2 &  &  \\
  \vdots &  & & \ddots & \vdots \\
& & & & a_{n-1} \\
a_{n} & & & \cdots & 0
\end{pmatrix},\,\,\,\,
\tilde{A} = \begin{pmatrix}
0 & a_1 &  & \cdots & & 0  \\
  &  & a_2 &  &  \\
  \vdots &  & & \ddots & & \vdots \\
& & & & a_{n-1} & \\
 & & &  & & a_{n} \\
 0 & & & \cdots & & 0
\end{pmatrix}
\end{equation} 
are related.  For any integer $n$, it is easy to verify that $\varepsilon(A) \leq \varepsilon(\tilde{A}) + 1$ by comparing the commutator matrices $[A,A^*]$ and $[\tilde{A},\tilde{A}^*]$ (these differ in only two columns).  If $n \leq 3$, we have $\nd(A) = \varepsilon(A)$ and $\nd(\tilde{A}) = \varepsilon(\tilde{A})$, so $\nd(A) \leq \nd(\tilde{A})+1$.  But does this inequality hold in general?
\begin{conj} \label{nd inequality}
Let $A \in \mathbb{C}^{n\times n}$, $\tilde{A}\in \mathbb{C}^{(n+1)\times (n+1)}$ with $n \geq 4$ be defined as in $(\ref{two matrices different size})$.  Then $\nd(\tilde{A}) \leq \nd(A) + 1$.
\end{conj}

The motivation for this conjecture is that, given a minimal normal completion matrix for $A$ of size $(n+k)\times (n+k)$, it is very easy to find a normal completion matrix of similar structure (not necessarily minimal) for $\tilde{A}$ of size $(n+1+k+1) \times (n+1+k+1)$.  The next example demonstrates this.

\begin{ex} \label{ex open question}
\rm{
The matrix $A = \begin{pmatrix}
0 & 2 & 0 & 0 \\
0 & 0 & 1 & 0 \\
0 & 0 & 0 & 1 \\
1 & 0 & 0 & 0
\end{pmatrix}$ has $\varepsilon(A)=1$.  But according to Proposition \ref{prop nd = max}, $\nd(A) > 1$.  An example of a minimal normal completion matrix of size $6 \times 6$ is
\begin{equation*}
 \left( \begin{array}{cccc|cc}
0 & 2 & 0 & 0 & 0 & 0 \\
0 & 0 & 1 & 0 & \sqrt{3} & 0 \\
0 & 0 & 0 & 1 & 0 & 0 \\
1 & 0 & 0 & 0 & 0 & -\sqrt{3} \\ \hline
0 & 0 & 0 & \sqrt{3} & 0 & 0 \\
\sqrt{3} & 0 & 0 & 0 & 0 & 1 
\end{array} \right),
\end{equation*}
and so we have $\nd(A)=2$.
We can use this to quickly find a normal completion of length $8\times8$ for the matrix $\tilde{A}$:
\begin{equation*}
 \left( \begin{array}{ccccc|ccc}
0 & 2 & 0 & 0 & 0 & 0 & 0 & 0 \\
0 & 0 & 1 & 0 & 0 & \sqrt{3} & 0 & 0 \\
0 & 0 & 0 & 1 & 0 & 0 & 0 & 0 \\
0 & 0 & 0 & 0 & 1 & 0 & -\sqrt{3} & 0 \\ 
0 & 0 & 0 & 0 & 0 & 0 & 0 & 2 \\ \hline
0 & 0 & 0 & \sqrt{3} & 0 & 0 & 0 & 0 \\
0 & 0 & 0 & 0 & \sqrt{3} & 0 & 1 & 0 \\
2 & 0 & 0 & 0 & 0 & 0 & 0 & 0  
\end{array} \right).
\end{equation*}
Thus, $\nd(\tilde{A}) \leq 3$.  Notice the similarity in the structure of the two completion matrices.  Also note that we have $\varepsilon(\tilde{A})=2$.  
}
\end{ex}

In the example just given, the matrix  $A$ had the property that $\nd(A)=\varepsilon(A)+1$.  It was observed earlier (\ref{nd unknown}) that the normal defect of $\tilde{A}$ is unknown.  If $\nd(\tilde{A})=3$, however, then $\tilde{A}$ would also have the property that $\nd(\tilde{A})=\varepsilon(\tilde{A})+1$.  This seems reasonable if indeed the normal defects of $A$ and $\tilde{A}$ are related.  We state this as a general conjecture:
\begin{conj}\label{conj nd = eps + 1}
Let $A \in \mathbb{C}^{n\times n}$, $\tilde{A}\in \mathbb{C}^{(n+1)\times (n+1)}$ with $n \geq 4$ be defined as in $(\ref{two matrices different size})$.  If  $\nd(A) = \varepsilon(A)+1$, then $\nd(\tilde{A}) = \varepsilon(\tilde{A})+1$.
\end{conj}
Note that for $n \leq 3$, we have that $\nd(A)=\varepsilon(A)$, so in that case the statement in Conjecture \ref{conj nd = eps + 1} is an empty statement.  If this conjecture is true, then it must be the case that the equality $\nd(\tilde{A})=\varepsilon(\tilde{A})$ does not hold in general for matrices $\tilde{A}$ of arbitrary size, with $\tilde{A}$ as in (\ref{two matrices different size}).

\section{Acknowledgements} 
Both authors were supported by NSF grant DMS 0901628.  Ryan D. Wasson performed the research as a Research Experience for Undergraduates (REU) project.

\section{Appendix}
Here we provide the details of the proof that the completion matrix given in (\ref{b_largest}) is normal (the proof for (\ref{c_largest}) and (\ref{a_largest}) will be omitted for the sake of brevity).  The entries of each matrix are real so we only need to check the equality $A_{ext}A_{ext}^T=A_{ext}^TA_{ext}$.  In this case we have $A_{ext}A_{ext}^T=$
\footnotesize
\begin{equation*}
 \left( \begin {array}{cccccc} {a}^{2}+{v_{{12}}}^{2}&0&0&v_{{12}}v_{{
42}}&aw_{{12}}+v_{{12}}z_{12}&0\\ \noalign{\medskip}0&{b}^{2}&0&0&0&0
\\ \noalign{\medskip}0&0&{c}^{2}+{v_{{31}}}^{2}&0&0&cw_{{24}}+v_{{31}}
z_{21}\\ \noalign{\medskip}v_{{12}}v_{{42}}&0&0&{v_{{42}}}^{2}&v_{{42
}}z_{12}&0\\ \noalign{\medskip}aw_{{12}}+v_{{12}}z_{12}&0&0&v_{{42}}
z_{12}&{w_{{12}}}^{2}+{z_{12}}^{2}&0\\ \noalign{\medskip}0&0&cw_{{24
}}+v_{{31}}z_{21}&0&0&{w_{{21}}}^{2}+{w_{{24}}}^{2}+{z_{21}}^{2}
\end {array} \right)
\end{equation*}
\normalsize
and $A_{ext}^TA_{ext}=$
\footnotesize
\begin{equation*}
 \left( \begin {array}{cccccc} {w_{{21}}}^{2}&0&0&w_{{21}}w_{{24}}&w_{
{21}}z_{21}&0\\ \noalign{\medskip}0&{a}^{2}+{w_{{12}}}^{2}&0&0&0&av_{
{12}}+w_{{12}}z_{12}\\ \noalign{\medskip}0&0&{b}^{2}&0&0&0
\\ \noalign{\medskip}w_{{21}}w_{{24}}&0&0&{c}^{2}+{w_{{24}}}^{2}&cv_{{
31}}+w_{{24}}z_{21}&0\\ \noalign{\medskip}w_{{21}}z_{21}&0&0&cv_{{31
}}+w_{{24}}z_{21}&{v_{{31}}}^{2}+{z_{21}}^{2}&0\\ \noalign{\medskip}0
&av_{{12}}+w_{{12}}z_{12}&0&0&0&{v_{{12}}}^{2}+{v_{{42}}}^{2}+{z_{12}}^{2}\end {array} \right). 
\end{equation*}
\normalsize
These are each symmetric, so we only need to consider the entries above and along the diagonal.  It is trivial to show that there is equality at the $(1,4)$, $(2,2)$, $(2,6)$, $(3,3)$, $(3,6)$, entries of each matrix above, so we will only consider the others.  First, $(1,1)$:
\begin{equation*}
a^2+v_{12}^2={a}^{2}+{\frac {{a}^{2} \left( {b}^{2}-{a}^{2} \right)  \left( {b}^{2}
-{c}^{2} \right) }{{b}^{2}{c}^{2}+{b}^{2}{a}^{2}-{a}^{2}{c}^{2}}}
={\frac {{a}^{2}{b}^{4}}{{b}^{2}{c}^{2}+{b}^{2}{a}^{2}-{a}^{2}{c}^{2}}} = w_{21}^2
\end{equation*}
Now $(1,5)$:
\begin{eqnarray*}
aw_{12}+v_{12}z_{12}&=&a\sqrt {{b}^{2}-{a}^{2}}-{a}^{3}{\frac {\sqrt{ \left( {b}^{2}-{a}^{2}
 \right)  \left( {b}^{2}-{c}^{2} \right)^2 }}{{b}^{2}{c}^{2}+{b}^{2}{a}^{
2}-{a}^{2}{c}^{2}}} \\
&=&\frac{ab^2c^2\sqrt{b^2-a^2}}{b^2c^2+b^2a^2-a^2c^2}\,\,\,\,=\,\,\,\,w_{21}z_{21}
\end{eqnarray*}
The entry $(4,4)$ is equivalent to $(1,1)$, but with $a$ and $c$ swapping roles.  Similarly, $(4,5)$ is equivalent to $(1,5)$.\\ \\
$(5,5)$:
\begin{eqnarray*}
w_{12}^2+z_{12}^2&=&{b}^{2}-{a}^{2}+{\frac {{a}^{4} \left( {b}^{2}-{c}^{2} \right) }{{b}^{
2}{c}^{2}+{b}^{2}{a}^{2}-{a}^{2}{c}^{2}}}\\
&=&\frac{b^4c^2+b^4a^2-2a^2b^2c^2}{{{b}^{
2}{c}^{2}+{b}^{2}{a}^{2}-{a}^{2}{c}^{2}}}\\
&=&\frac{\left(b^2-c^2\right)\left(b^2c^2+b^2a^2-a^2c^2\right)+c^2(b^2c^2+b^2a^2-a^2c^2)-a^2b^2c^2}{b^2c^2+b^2a^2-a^2c^2}\\
&=&b^2-c^2+\frac{c^4\left(b^2-a^2\right)}{b^2c^2+b^2a^2-a^2c^2}\,\,\,\,=\,\,\,\,v_{31}^2+z_{21}^2
\end{eqnarray*}
$(6,6)$:
\begin{eqnarray*}
w_{21}^2+w_{24}^2+z_{21}^2&=&\frac{a^2b^4+c^2\left(b^2-a^2\right)\left(b^2-c^2\right)+c^4\left(b^2-a^2\right)}{b^2c^2+b^2a^2-a^2c^2}\\
&=&\frac{b^2\left(b^2c^2+b^2a^2-a^2c^2\right)}{b^2c^2+b^2a^2-a^2c^2}\,\,\,\,=\,\,\,\,b^2
\end{eqnarray*}
and
\begin{eqnarray*}
v_{12}^2+v_{42}^2+z_{12}^2&=&\frac{a^2\left(b^2-a^2\right)\left(b^2-c^2\right)+c^2b^4+a^4\left(b^2-c^2\right)}{b^2c^2+b^2a^2-a^2c^2}\\
&=&\frac{b^2\left(b^2c^2+b^2a^2-a^2c^2\right)}{b^2c^2+b^2a^2-a^2c^2}\,\,\,\,=\,\,\,\,b^2
\end{eqnarray*}
Therefore (\ref{b_largest}) is normal.

\end{document}